\documentclass{amsart}
\usepackage{breqn}
\usepackage{placeins}
\usepackage{float} 
\newtheorem{theorem}{Theorem}[section]

\newtheorem{proposition}[theorem]{Proposition}

\theoremstyle{definition}

\theoremstyle{remark}

\numberwithin{equation}{section}

\begin{document}

\title[Double integral of logarithm and exponential function]{Double integral of logarithm and exponential function expressed in terms of the Lerch function}


\author{Robert Reynolds}
\address{Department of Mathematics and Statistics, York University, Faculty of Science, York University}
\email{milver@my.yorku.ca}

\author{Allan Stauffer}
\address{Department of Mathematics and Statistics, York University, Faculty of Science, York University}
\email{stauffer@yorku.ca}



\date{}

\dedicatory{}

\begin{abstract}
In this work the authors use their contour integral method to derive a double integral connected to the modified Bessel function of the second kind and express it in terms of the Lerch function. There are some useful results relating double integrals of certain kinds of functions to ordinary integrals for which we know no general reference. Thus a table of integral pairs is given for interested readers. The majority of the results in this work are new.
\end{abstract}

\maketitle

\section{Introduction}

In 1986 A.P. Prudnikov \cite{prud} produced volume 1 of his five volume collection on Integrals and Series. In this work, the authors used their contour integral method and applied it to an interesting integral in the book of Prudnikov \cite{prud} and expressed its closed form in terms of the Lerch function. Double integrals play a significant role in the area mathematics, namely in the evaluation of moment of inertia, centre of mass, volumes of solids of revolution, averages and in error analysis of integral and discrete transforms \cite{jerri}. For certain values of the parameters the integrand of this derived integral formula involved the modified Bessel function of the second kind $K_{n}(z)$. This integral formula was then used to provide formal derivations in terms of  and new formulae in the form of a summary table of integrals. The Lerch function being a special function has the fundamental property of analytic continuation, which enables us to widen the range of evaluation for the parameters involved in our definite integral.  

The definite integral derived in this manuscript is given by
 
 \begin{equation}
\int_{0}^{\infty} \int_{0}^{\infty}x^m y^{-m-1} e^{-p x-q y-\frac{x^2}{4 y}} \log ^k\left(\frac{a x}{y}\right)dxdy
  \end{equation}

where the parameters $k$, $a$, $p$ and $q$ are general complex numbers and $-1<Re(m)\leq -1/2, -1<Im(m)<-1/2$. This work is important because the authors were unable to find similar derivations in current literature. The derivation of the definite integral follows the method used by us in \cite{reyn4} which involves Cauchy's integral formula. The generalized Cauchy's integral formula is given by

\begin{equation}\label{intro:cauchy}
\frac{y^k}{k!}=\frac{1}{2\pi i}\int_{C}\frac{e^{wy}}{w^{k+1}}dw.
\end{equation}

where $C$ is in general an open contour in the complex plane where the bilinear concomitant has the same value at the end points of the contour. This method involves using a form of equation (\ref{intro:cauchy}) then multiply both sides by a function, then take a definite integral of both sides. This yields a definite integral in terms of a contour integral. A second contour integral is derived by multiplying equation (\ref{intro:cauchy}) by a function and performing some substitutions so that the contour integrals are the same.

\section{Definite integral of the contour integral}

We use the method in \cite{reyn4}. The cut and contour are in the second quadrant of the complex $z$-plane.  The cut approaches the origin from the interior of the second quadrant and the contour goes round the origin with zero radius and is on opposite sides of the cut. Using equation (\ref{intro:cauchy}) we replace $y$ by $\log(\frac{ax}{y})$ then multiply by $x^m y^{-m-1} e^{-p x-q y-\frac{x^2}{4 y}}$. Next we take the double infinite integral over $x\in(0,\infty)$ and $y\in(0,\infty)$ to get

\begin{dmath}\label{dici}
\frac{1}{k!}\int_{0}^{\infty}\int_{0}^{\infty}x^m y^{-m-1} e^{-p x-q y-\frac{x^2}{4 y}} \log ^k\left(\frac{a x}{y}\right)dxdy
=\frac{1}{2\pi i}\int_{0}^{\infty}\int_{0}^{\infty}\int_{C}a^w w^{-k-1} x^{m+w} y^{-m-w-1} e^{-p x-q y-\frac{x^2}{4 y}}dwdxdy
=\frac{1}{2\pi i}\int_{C}\int_{0}^{\infty}\int_{0}^{\infty}a^w w^{-k-1} x^{m+w} y^{-m-w-1} e^{-p x-q y-\frac{x^2}{4 y}}dxdydw
=\frac{1}{2\pi i}\int_{C}\frac{\pi  a^w w^{-k-1} 2^{m+w+1} q^{\frac{m+w}{2}} \csc (\pi  (m+w)) \sinh
   \left((m+w) \cosh ^{-1}\left(\frac{p}{\sqrt{q}}\right)\right)}{\sqrt{p^2-q}}dw
\end{dmath}

from equation (3.1.3.48) in \cite{prud} where $|Re(m+w)|<1$. We are able to switch the order of integration over $z$, $x$ and $y$ using Fubini's theorem since the integrand is of bounded measure over the space $C \times \mathbb{R} \times \mathbb{R}$.

\section{The Lerch function}

We use (9.550) and (9.556) in \cite{grad} where $\Phi(z,s,v)$ is the Lerch function which is a generalization of the Hurwitz zeta $\zeta(s,v)$ and Polylogarithm functions $Li_{n}(z)$. The Lerch function has a series representation given by

\begin{equation}\label{knuth:lerch}
\Phi(z,s,v)=\sum_{n=0}^{\infty}(v+n)^{-s}z^{n}
\end{equation}

where $|z|<1, v \neq 0,-1,..$ and is continued analytically by its integral representation given by

\begin{equation}\label{knuth:lerch1}
\Phi(z,s,v)=\frac{1}{\Gamma(s)}\int_{0}^{\infty}\frac{t^{s-1}e^{-vt}}{1-ze^{-t}}dt=\frac{1}{\Gamma(s)}\int_{0}^{\infty}\frac{t^{s-1}e^{-(v-1)t}}{e^{t}-z}dt
\end{equation}

where $Re(v)>0$, or $|z| \leq 1, z \neq 1, Re(s)>0$, or $z=1, Re(s)>1$.

\section{Infinite sum of the contour integral}

In this section we will again use Cauchy's integral formula (\ref{intro:cauchy}) and taking the infinite sum to derive equivalent sum representations for the contour integrals. We proceed using equation (\ref{intro:cauchy}) to form two equations and take their difference. Firstly, replace $y\to y+x$ and multiply both sides by $e^{mx}$ and secondly, replace $x\to -x$ for the second and subtract to get

\begin{dmath}
\frac{e^{m x} (x+y)^k-e^{-m x} (y-x)^k}{2 k!}=\frac{1}{2\pi i}\int_{C}w^{-k-1} e^{w y} \sinh (x (m+w))dw
\end{dmath}

Next we replace $y \to \log (a)+\frac{\log (q)}{2}+i \pi  (2 y+1)+\log (2)$ and multiply both sides by $e^{i \pi  m (2 y+1)}$ and take the infinite sum over $y\in [0,\infty)$ to get

\begin{dmath}
\frac{2^{k-1} \pi ^k e^{-m x+\frac{1}{2} i \pi  (k+2 m)}}{k!} \left(e^{2 m x} \Phi \left(e^{2 i m \pi },-k,-\frac{2 i x+2 i \log (2 a)+i \log
   (q)-2 \pi }{4 \pi }\right)-\Phi \left(e^{2 i m \pi },-k,\frac{2 i x-2 i \log (2 a)-i \log (q)+2 \pi }{4 \pi }\right)\right)
   =\frac{1}{2\pi i}\sum_{y=0}^{\infty}\int_{C}2^w a^w w^{-k-1} q^{w/2} e^{i \pi  (2 y+1) (m+w)} \sinh (x
   (m+w))dw
   =\frac{1}{2\pi i}\int_{C}\sum_{y=0}^{\infty}2^w a^w w^{-k-1} q^{w/2} e^{i \pi  (2 y+1) (m+w)} \sinh (x
   (m+w))dw
   =\frac{1}{2\pi i}\int_{C}i 2^{w-1}
   a^w w^{-k-1} q^{w/2} \csc (\pi  (m+w)) \sinh (x (m+w))dw
\end{dmath}

from equation (1.232.3) in \cite{grad} where $Im(m+w)<0$.
Next we multiply by $-\frac{i \pi  2^{m+2} q^{m/2}}{\sqrt{p^2-q}}$ and replace $x\to \cosh ^{-1}\left(\frac{p}{\sqrt{q}}\right)$ simplifying to get

\begin{dmath}\label{isci}
\left(e^{2 m \cosh
   ^{-1}\left(\frac{p}{\sqrt{q}}\right)} \Phi \left(e^{2 i m \pi },-k,-\frac{2 i \cosh ^{-1}\left(\frac{p}{\sqrt{q}}\right)+2 i \log (2 a)+i
   \log (q)-2 \pi }{4 \pi }\right)-\Phi \left(e^{2 i m \pi },-k,\frac{2 i \cosh ^{-1}\left(\frac{p}{\sqrt{q}}\right)-2 i \log (2 a)-i \log
   (q)+2 \pi }{4 \pi }\right)\right)
   =\frac{1}{2\pi i}\int_{C}\frac{\pi  a^w w^{-k-1} 2^{m+w+1} q^{\frac{m}{2}+\frac{w}{2}} \csc (\pi  (m+w)) \sinh
   \left((m+w) \cosh ^{-1}\left(\frac{p}{\sqrt{q}}\right)\right)}{\sqrt{p^2-q}}dw
\end{dmath}

where $\bar{s}=-\frac{i \pi ^{k+1} 2^{k+m+1} q^{m/2} e^{-m \cosh ^{-1}\left(\frac{p}{\sqrt{q}}\right)+\frac{1}{2} i \pi  (k+2 m)}}{k! \sqrt{p^2-q}}$.

\section{Double integral in terms of the Lerch function}

\begin{theorem}
For $k,a,p,q \in\mathbb{C}, -1<Re(m)\leq -1/2, -1<Im(m)<-1/2$,
\begin{dmath}\label{dilf}
\int_{0}^{\infty}\int_{0}^{\infty}x^m y^{-m-1} e^{-p x-q y-\frac{x^2}{4 y}} \log ^k\left(\frac{a x}{y}\right)dxdy
=\bar{r}\left(e^{2 m \cosh ^{-1}\left(\frac{p}{\sqrt{q}}\right)} \Phi \left(e^{2 i
   m \pi },-k,-\frac{2 i \cosh ^{-1}\left(\frac{p}{\sqrt{q}}\right)+2 i \log (2 a)+i \log (q)-2 \pi }{4 \pi }\right)-\Phi \left(e^{2 i m \pi
   },-k,\frac{2 i \cosh ^{-1}\left(\frac{p}{\sqrt{q}}\right)-2 i \log (2 a)-i \log (q)+2 \pi }{4 \pi }\right)\right)
\end{dmath}

where $\bar{r}=k!\bar{s}=-\frac{i \pi ^{k+1} 2^{k+m+1} q^{m/2} e^{-m \cosh
   ^{-1}\left(\frac{p}{\sqrt{q}}\right)+\frac{1}{2} i \pi  (k+2 m)} }{\sqrt{p^2-q}}$
\end{theorem}

\begin{proof}
Since the right-hand sides of equations (\ref{dici}) and (\ref{isci}) are equal we can equate the left-hand sides and simplify the factorial to get the stated result.
\end{proof}

\section*{Main Results}

\section{Derivation of entry 3.1.3.48 in \cite{prud}}

\begin{proposition}
For $p,q \in\mathbb{C}, -1<Re(m)\leq -1/2, -1<Im(m)<-1/2$,
\begin{dmath}
\int_{0}^{\infty}\int_{0}^{\infty}x^m y^{-m-1} e^{-p x-q y-\frac{x^2}{4 y}}dxdy=\frac{\pi  2^{m+1} q^{m/2} \csc (\pi  m) \sinh \left(m \cosh
   ^{-1}\left(\frac{p}{\sqrt{q}}\right)\right)}{\sqrt{p^2-q}}
\end{dmath}
\end{proposition}

\begin{proof}
Use equation (\ref{dilf}) and set $k=0$ and simplify using entry (2) in table below (64:12:7) in \cite{atlas}.
\end{proof}

\section{Derivation of new entry 3.1.3.59 in \cite{prud}}

\begin{proposition}
\begin{dmath}
\int_{0}^{\infty}\int_{0}^{\infty}x^m y^{-m-1} \log \left(\frac{a x}{y}\right) e^{-p x-q y-\frac{x^2}{4 y}}dxdy
=-\frac{\pi  2^m q^{m/2} \csc (\pi  m) }{\sqrt{p^2-q}}\left((-2 \log (a)+2 \pi
    \cot (\pi  m)-\log (4 q)) \sinh \left(m \cosh ^{-1}\left(\frac{p}{\sqrt{q}}\right)\right)-2 \cosh ^{-1}\left(\frac{p}{\sqrt{q}}\right)
   \cosh \left(m \cosh ^{-1}\left(\frac{p}{\sqrt{q}}\right)\right)\right)
\end{dmath}
\end{proposition}

\begin{proof}
Use equation (\ref{dilf}) and set $k=1$ and simplify using entry (1) in table below (64:12:7) in \cite{atlas}.
\end{proof}

\section{Derivation of new entry 3.1.3.60 in \cite{prud}}

In this example we look at the Laplace transform of the first partial derivative with respect to $\nu$ of the modified Bessel function of the second kind $K_{\nu}(x)$ where $\nu=-1/2,x=x/2$. Details about this function are listed in equation (51:4:1) in \cite{atlas} and equation (11.123a) in \cite{weber}.

\begin{proposition}
\begin{dmath}
\int_{0}^{\infty}\int_{0}^{\infty}\frac{e^{-\frac{x^2+4 x y+y^2}{4 y}} \log \left(\frac{x}{y}\right)}{\sqrt{x} \sqrt{y}}dxdy=-2 \sqrt{2} \pi  \cosh ^{-1}(2)
\end{dmath}
\end{proposition}

\begin{proof}
Use equation (\ref{dilf}) and set $m=-1/2,a=1$ and simplify in terms of the Hurwitz zeta function using entry (4) in table below (64:12:7) in \cite{atlas}. Next set $k=p=1,q=1/4$ and simplify using entry (2) in table below (64:4:1) in \cite{atlas}.
\end{proof}

\section{Derivation of new entry 3.1.3.61 in \cite{prud}}

\begin{proposition}
\begin{dmath}
\int_{0}^{\infty}\int_{0}^{\infty}\frac{e^{-\frac{x^2+4 x y+y^2}{4 y}} \log ^2\left(\frac{x}{y}\right)}{\sqrt{x} \sqrt{y}}dxdy=2 \sqrt{\frac{2}{3}} \pi  \left(\pi ^2+(\cosh^{-1}(2))^2\right)
\end{dmath}
\end{proposition}

\begin{proof}
Use equation (\ref{dilf}) and set $m=-1/2,a=1$ and simplify in terms of the Hurwitz zeta function using entry (4) in table below (64:12:7) in \cite{atlas}. Next set $k=2,p=1,q=1/4$ and simplify using entry (3) in table below (64:4:1) in \cite{atlas}.
\end{proof}

\section{Derivation of new entry 3.1.3.62 in \cite{prud}}

\begin{proposition}
For $k\in\mathbb{C}$,
\begin{dmath}\label{prop10}
\int_{0}^{\infty}\int_{0}^{\infty}\frac{e^{-\frac{x^2+x y+y^2}{4 y}} (x-y) \log ^k\left(\frac{x}{y}\right)}{\sqrt{x} y^{3/2}}dxdy=i 2^{2 k+3} e^{\frac{i \pi  k}{2}} \pi
   ^{k+1} \left(\zeta \left(-k,\frac{1}{6}\right)+\zeta \left(-k,\frac{5}{6}\right)+\left(1-3^{-k}\right) \zeta (-k)\right)
\end{dmath}
\end{proposition}

\begin{proof}
Use equation (\ref{dilf}) and set $a=1,m=-1/2,q=p$ and simplify in terms of the Hurwitz zeta functions using entry (4) in table below (64:12:7) in \cite{atlas}. Next set $p=1/4$ and simplify using entry (2) in table below (64:7) in \cite{atlas}.
\end{proof}

\section{Derivation of new entry 3.1.3.63 in \cite{prud}}

\begin{proposition}
\begin{dmath}
\int_{0}^{\infty}\int_{0}^{\infty}\frac{e^{-\frac{x^2+x y+y^2}{4 y}} (x-y) \log \left(\frac{x}{y}\right) \log \left(\log \left(\frac{x}{y}\right)\right)}{\sqrt{x}
   y^{3/2}}dxdy\\
   =\frac{4}{9} \pi ^2 \left(6+3 \pi  i+\log \left(\frac{2^{14} 3^3 \pi ^6}{A^{72}}\right)\right)
\end{dmath}
\end{proposition}

\begin{proof}
Use equation (\ref{prop10}) and take the first partial derivative with respect to $k$ then set $k=1$ and simplify using equations (A11) and (A12) in \cite{voros}.
\end{proof}

\section{Derivation of new entry 3.1.3.64 in \cite{prud}}

\begin{proposition}
\begin{dmath}
\int_{0}^{\infty}\int_{0}^{\infty}\frac{e^{-\frac{x^2+x y+y^2}{4 y}} \left(\sqrt[4]{y}-\sqrt[4]{x}\right)}{x^{3/4} \sqrt{y} \log \left(\frac{x}{y}\right)}dxdy=4 \log
   \left(4-2 \sqrt{3}\right)
\end{dmath}
\end{proposition}

\begin{proof}
Use equation (\ref{dilf}) and replace $m\to t$ to form a second equation and take their difference. Next set $m=-1/2,t=-3/4,a=1,q=p$ and simplify the right-hand side in terms of the Hurwitz zeta function using entry (4) in table below (64:12:7) in \cite{atlas}. Next apply L'Hopital's rule to the right-hand side as $k\to-1$ and simplify using entry (1) in table below (64:12:7) and equation (64:4:1) in \cite{atlas}.
(\end{proof}

\section{Derivation of new entry 3.1.3.65 in \cite{prud}}

\begin{proposition}
For $-1<Re(m)\leq -1/2, -1<Im(m)<-1/2,-1<Re(t)\leq -1/2, -1<Im(t)<-1/2$,
\begin{dmath}\label{eq2f1}
\int_{0}^{\infty}\int_{0}^{\infty}\frac{e^{-\frac{x^2+x y+y^2}{4 y}} y^{-m-t-1} \left(y^m x^t-x^m y^t\right)}{\log \left(\frac{x}{y}\right)}dxdy
=2 i \sqrt{3} \left(2
   e^{\frac{2 i \pi  m}{3}} \, _2F_1\left(\frac{1}{3},1;\frac{4}{3};e^{2 i m \pi }\right)-e^{\frac{4 i \pi  m}{3}} \,
   _2F_1\left(\frac{2}{3},1;\frac{5}{3};e^{2 i m \pi }\right)-2 e^{\frac{2 i \pi  t}{3}} \, _2F_1\left(\frac{1}{3},1;\frac{4}{3};e^{2 i \pi 
   t}\right)+e^{\frac{4 i \pi  t}{3}} \, _2F_1\left(\frac{2}{3},1;\frac{5}{3};e^{2 i \pi  t}\right)\right)
\end{dmath}
\end{proposition}

\begin{proof}
Use equation (\ref{dilf}) and set $k=-1,a=1,p=q=1/4$ and simplify using (9.559) in \cite{grad}.
\end{proof}

\section{Derivation of new entry 3.1.3.66 in \cite{prud}}

\begin{proposition}
\begin{dmath}
\int_{0}^{\infty}\int_{0}^{\infty}\frac{e^{-\frac{x^2+x y+y^2}{4 y}} \left(\sqrt[6]{y}-\sqrt[6]{x}\right)}{x^{2/3} \sqrt{y} \log \left(\frac{x}{y}\right)}dxdy=\log
   \left(\frac{\sec ^4\left(\frac{\pi }{9}\right)}{4 \left(\sin \left(\frac{\pi }{36}\right)+\cos \left(\frac{\pi
   }{36}\right)\right)^4}\right)
\end{dmath}
\end{proposition}

\begin{proof}
Use equation (\ref{eq2f1}) and set $m=-1/2,t=-2/3$ and simplify using (9.559) in \cite{grad}.
\end{proof}

\section{Derivation of new entry 3.1.3.67 in \cite{prud}}

\begin{proposition}
\begin{dmath}
\int_{0}^{\infty}\int_{0}^{\infty}\frac{e^{-\frac{x^2+x y+y^2}{4 y}} \left(\sqrt[12]{y}-\sqrt[12]{x}\right)}{x^{3/4} \sqrt[3]{y} \log \left(\frac{x}{y}\right)}dxdy=2
   \left(\log \left(\frac{7}{4}-\sqrt{3}\right)+2 \log \left(\csc \left(\frac{\pi }{18}\right)\right)\right)
\end{dmath}
\end{proposition}

\begin{proof}
Use equation (\ref{eq2f1}) and set $m=-2/3,t=-3/4$ and simplify using (9.559) in \cite{grad}.
\end{proof}

\section{Derivation of new entry 3.1.3.68 in \cite{prud}}

\begin{proposition}
\begin{dmath}
\int_{0}^{\infty}\int_{0}^{\infty}\frac{e^{-\frac{x^2+x y+y^2}{4 y}} \left(\sqrt[6]{y}-\sqrt[6]{x}\right)}{\sqrt{x} y^{2/3} \log \left(\frac{x}{y}\right)}dxdy=2 \log
   \left(\frac{1+\cos \left(\frac{\pi }{9}\right)}{4-4 \sin \left(\frac{\pi }{18}\right)}\right)
\end{dmath}
\end{proposition}

\begin{proof}
Use equation (\ref{eq2f1}) and set $m=-1/3,t=-1/2$ and simplify using (9.559) in \cite{grad}.
\end{proof}

\section{Derivation of new entry 3.1.3.69 in \cite{prud}}

\begin{proposition}
\begin{dmath}
\int_{0}^{\infty}\int_{0}^{\infty}\frac{e^{-\frac{x^2+x y+y^2}{4 y}} \left(\sqrt[4]{y}-\sqrt[4]{x}\right)}{x^{3/4} \sqrt{y} \log \left(\frac{x}{y}\right)}dxdy=4 \log
   \left(4-2 \sqrt{3}\right)
\end{dmath}
\end{proposition}

\begin{proof}
Use equation (\ref{eq2f1}) and set $m=-1/2,t=-3/4$ and simplify using (9.559) in \cite{grad}.
\end{proof}

\section{Derivation of new entry 3.1.3.70 in \cite{prud}}

In this section we will look at the limiting case of equation \ref{dilf} when $p=q$ and apply L'Hopitals' rule as $q\to 1$ and express the integral in terms of the Hypergeometric function section (9.1) and equation (9.559) in \cite{grad}.

\begin{proposition}
For $-1<Re(m)\leq -1/2, -1<Im(m)<-1/2,-1<Re(t)\leq -1/2, -1<Im(t)<-1/2$,
\begin{dmath}
\int_{0}^{\infty}\int_{0}^{\infty}\frac{e^{-\frac{(x+2 y)^2}{4 y}} y^{-m-t-1} \left(y^m x^t-x^m y^t\right)}{\log \left(\frac{x}{y}\right)}dxdy
=\frac{2^m e^{i \pi  m} }{\pi }\left(2
   \pi  m \Phi \left(e^{2 i m \pi },1,\frac{\pi -i \log (2)}{2 \pi }\right)+i \Phi \left(e^{2 i m \pi },2,\frac{\pi -i \log (2)}{2 \pi
   }\right)\right)-2^t e^{i \pi  t} \left(2 \pi  t \Phi \left(e^{2 i \pi  t},1,\frac{\pi -i \log (2)}{2 \pi }\right)+i \Phi \left(e^{2 i \pi 
   t},2,\frac{\pi -i \log (2)}{2 \pi }\right)\right)
\end{dmath}
\end{proposition}

\begin{proof}
Use equation (\ref{dilf}) and set $p=q4$ and apply L'Hopital's' rule to the right-hand side simplify. Next form a second integral by replacing $m\to t$ and take their difference.
\end{proof}

 \section{Summary table of results}
 
 \renewcommand{\arraystretch}{2.8}
 \FloatBarrier
\begin{table}[H]\vspace*{-3ex}
\begin{tabular}{ l  c }
  \hline			
  $f(x,y)$ & $\int_{0}^{\infty}\int_{0}^{\infty}f(x,y)dxdy$ \\ \hline
  $x^m y^{-m-1} e^{-p x-q y-\frac{x^2}{4 y}}$ & $\frac{\pi  2^{m+1} q^{m/2} \csc (\pi  m) \sinh \left(m \cosh
   ^{-1}\left(\frac{p}{\sqrt{q}}\right)\right)}{\sqrt{p^2-q}}$ \\ 
  $\frac{e^{-\frac{x^2+4 x y+y^2}{4 y}} \log ^2\left(\frac{x}{y}\right)}{\sqrt{x} \sqrt{y}}$ & $2 \sqrt{\frac{2}{3}} \pi  \left(\pi ^2+\cosh
   ^{-1}(2)^2\right)$ \\
  $\frac{e^{-\frac{x^2+x y+y^2}{4 y}} (x-y) \log ^k\left(\frac{x}{y}\right)}{\sqrt{x} y^{3/2}}$ & $i 2^{2 k+3} e^{\frac{i \pi  k}{2}} \pi
   ^{k+1} \left(\zeta \left(-k,\frac{1}{6}\right)+\zeta \left(-k,\frac{5}{6}\right)+\left(1-3^{-k}\right) \zeta (-k)\right)$ \\
  $\frac{e^{-\frac{x^2+x y+y^2}{4 y}} (x-y) \log \left(\frac{x}{y}\right) \log \left(\log \left(\frac{x}{y}\right)\right)}{\sqrt{x}
   y^{3/2}}$ & $\frac{4}{9} \pi ^2 \left(6+3 \pi  i+\log \left(\frac{2^{14} 3^3 \pi ^6}{A^{72}}\right)\right)$ \\
  $\frac{e^{-\frac{x^2+x y+y^2}{4 y}} \left(\sqrt[4]{y}-\sqrt[4]{x}\right)}{x^{3/4} \sqrt{y} \log \left(\frac{x}{y}\right)}$ & $4 \log
   \left(4-2 \sqrt{3}\right)$ \\
  $\frac{e^{-\frac{x^2+x y+y^2}{4 y}} \left(\sqrt[6]{y}-\sqrt[6]{x}\right)}{x^{2/3} \sqrt{y} \log \left(\frac{x}{y}\right)}$ & $\log
   \left(\frac{\sec ^4\left(\frac{\pi }{9}\right)}{4 \left(\sin \left(\frac{\pi }{36}\right)+\cos \left(\frac{\pi
   }{36}\right)\right)^4}\right)$ \\
  $\frac{e^{-\frac{x^2+x y+y^2}{4 y}} \left(\sqrt[12]{y}-\sqrt[12]{x}\right)}{x^{3/4} \sqrt[3]{y} \log \left(\frac{x}{y}\right)}$ & $2
   \left(\log \left(\frac{7}{4}-\sqrt{3}\right)+2 \log \left(\csc \left(\frac{\pi }{18}\right)\right)\right)$ \\
  $\frac{e^{-\frac{x^2+x y+y^2}{4 y}} \left(\sqrt[6]{y}-\sqrt[6]{x}\right)}{\sqrt{x} y^{2/3} \log \left(\frac{x}{y}\right)}$ & $2 \log
   \left(\frac{1+\cos \left(\frac{\pi }{9}\right)}{4-4 \sin \left(\frac{\pi }{18}\right)}\right)$ \\
  $\frac{e^{-\frac{x^2+x y+y^2}{4 y}} \left(\sqrt[4]{y}-\sqrt[4]{x}\right)}{x^{3/4} \sqrt{y} \log \left(\frac{x}{y}\right)}$ & $4 \log
   \left(4-2 \sqrt{3}\right)$
      \\[0.3cm]
  \hline  
\end{tabular}
\end{table}
 \FloatBarrier

\section{Discussion}

In this work the authors derived a double integral formula in terms of the Lerch function. This integral formula was then used to derive special cases in terms of fundamental constants and special functions. A table of integrals featuring some of the integral results was presented for the benefit of interested readers. We used Wolfram Mathematica to numerically verify the formulas for various ranges of the parameters for real and imaginary values. We will use our contour integral method to derive other double integrals and produce more tables of integrals in our future work.

\end{document}